\documentclass{degruyter-journal-a}           % OTHER JOURNAL

 \received{April 29, 2013}
\usepackage{amsmath,amssymb} % for \begin{bmatrix}, \begin{smallmatrix}, \begin{align} etc.
\usepackage{amsfonts} % for \mathbb{} etc. `blackboard' for uppercase only.
\usepackage{exscale}

\usepackage{kbordermatrix}

\usepackage{graphicx} 
\usepackage{mydoi}

\renewcommand{\ldots}{\ensuremath{\dotsc}}

\def\F{{\mathcal F}}
\newcommand{\Ra}{\mathcal{R}} % range
\newcommand{\C}{\mathbb{C}} % complex 
\def\Z{{\mathcal Z}}
\def\H{{\mathcal H}}
\newcommand{\X}{\mathcal{X}}
\newcommand{\Y}{\mathcal{Y}}
\newcommand{\XP}{{\mathcal{X}^\perp}}
\newcommand{\YP}{{\mathcal{Y}^\perp}}
\newcommand{\QX}{X}
\newcommand{\QXP}{X_\perp}
\newcommand{\QY}{Y}
\newcommand{\QYP}{Y_\perp}
\newcommand{\QZ}{Z}
\newcommand{\diag}{\mathop{\mathrm{diag}}} 
\newcommand{\rank}{\mathop{\mathrm{rank}}} 
\newcommand{\an}{\mathop{\mathrm{and}}}

\def\dim{{{\rm dim}}}
\def\max{{{\rm max}}}
\def\min{{{\rm min}}}

\newcommand{\moo}{\mathfrak{M}_{00}}
\newcommand{\moi}{\mathfrak{M}_{01}}
\newcommand{\mio}{\mathfrak{M}_{10}}
\newcommand{\mii}{\mathfrak{M}_{11}}
\newcommand{\mm}{\mathfrak{M}}

\newtheorem{thm}{Theorem}[section]
\newtheorem{defin}{Definition}[section]

\newtheorem{cor}{Corollary}[section]

\title{Angles between subspaces and their tangents
\footnote{%
A preprint is uploaded at \url{http://arxiv.org/abs/1209.0523}. 
}}
\headlinetitle{Angles between subspaces and their tangents}

\lastnameone{Zhu}
\firstnameone{Peizhen}
\nameshortone{P.~Zhu}
\addressone{Department of Mathematical and Statistical Sciences, 
University of Colorado Denver.}
\countryone{USA}
\emailone{peizhen.zhu[at]ucdenver.edu}

\lastnametwo{Knyazev}
\firstnametwo{Andrew V.}
\nameshorttwo{A.\,V.~Knyazev}
\addresstwo{Mitsubishi Electric Research Laboratories; 201 Broadway
Cambridge, MA 02139}
\countrytwo{USA}
\emailtwo{andrew.knyazev[at]merl.com}

\abstract{
Principal angles between subspaces (PABS) (also called canonical angles) serve as a classical tool   
in mathematics, statistics, and applications, e.g.,\ data mining.
Traditionally, PABS are introduced via their cosines.
The cosines and sines of PABS are commonly 
defined using the singular value decomposition. 
We utilize the same idea for the tangents, i.e.,\ explicitly construct matrices, such that
their singular values are equal to the tangents of PABS, 
using several approaches:
orthonormal and non-orthonormal bases for subspaces, as well as projectors. 
Such a construction has applications, e.g.,\ 
in analysis of convergence of subspace iterations for eigenvalue problems.
}

\researchsupported{This material is based upon work partially supported by the National Science Foundation under Grant No. 1115734. }

\acknowledgments{We thank Merico Argentati, Ilse Ipsen, and Dianne O'Leary for reading 
the preliminary version of this paper and for the feedback that helped 
us to improve the presentation.}

\begin{document}

\section{Introduction}\label{intro}
The concept of  principal angles between subspaces (PABS) 
is introduced by Jordan \cite{Jordan75} in 1875. 
Hotelling \cite{hotelling36} defines PABS in the form of
canonical correlations in statistics in 1936.  
Traditionally, PABS are introduced and used via their 
sines and more commonly, because of their connection to canonical correlations, cosines;
see, e.g.,\  
\cite{{Deutsch94},{Ipsen95},{knyazev06},{sunste},{wedin83}}.
The properties of sines and cosines of PABS are well investigated; 
e.g.,\ in \cite{{bjorck73},{ka02},{sun87}}. 

The tangents of PABS have attracted relatively less attention, compared to the cosines, 
despite of the celebrated work 
of Davis and Kahan \cite{kahan}, which includes several tangent-related theorems.  
The tangents of PABS also appear in several other important publications on numerical matrix analysis.  
In \cite{{bosner09},{drmac96}}, the authors use the tangent of the largest 
 principal angle derived from a norm of a specific matrix. 
In~\cite[Theorem 2.4, p. 252]{Stewart01} and \cite[p. 231-232]{sunste}
the tangents of PABS, related to singular values 
of a matrix---without an explicit matrix formulation---are used 
to analyze perturbations of invariant subspaces. 
The tangents of PABS are used in \cite{drma98} for generalized singular value computation. 
Properties of an oblique projector (idempotent) are naturally determined by 
the tangents of angles between its null space and range, e.g.,\
the spectral norm of an idempotent is the secant of the largest canonical angle between its row
and column spaces; see \cite{Stewart11} and references there. 

In~\cite{{bosner09},{drmac96},{Stewart01},{Stewart11},{sunste}}, 
the two subspaces have the same dimensions. 
In this work, for two given subspaces (not necessarily of the same dimensions)
we construct a family $\F$ of explicitly given matrices, 
such that the singular values of the matrix $T\in\F$ are the tangents of PABS. 
We find $T$ in two different ways. First, we derive $T$ using 
matrices whose columns form the (orthonormal) bases of the subspaces. Second, we present  
$T$ as the product of projectors on the subspaces. 
We~describe the action of  the matrix $T$ as a linear operator, and provide a geometric interpretation of
the singular values of $T$. Some of our proofs are reasonably technical, although rather standard 
and should be accessible to a wide audience.

Our basis-based constructions include the 
matrices used in~\cite{{bosner09},{drmac96},{Stewart01},{sunste}}
and extend to the case of subspaces with different dimensions.
These results are motivated by and have applications in 
our recent and upcoming work on majorization-based analysis of convergence of the Rayleigh-Ritz method \cite[Section 2.5]{ka09} and subspace iterations for eigenvalue problems~\cite{Zhutan}.
Let us, however, warn a non-expert reader that the tangent function is 
not well-suited for computations of angles close to $\pi/2$ for evident 
reasons. Thus, our formulas for $T$ cannot be recommended in general for numerical evaluation 
of PABS in their full range of values $[0,\pi/2].$

The basis-based approach bounds us to the matrix theory, as we represent subspaces by matrices, 
e.g.,\ it makes it difficult to attempt extending our results to general infinite-dimensional Hilbert spaces 
by analogy with \cite{{kja10}}. Our projector-based results are more intuitive and general, 
relying only on geometry of the space.  
However, we do not develop here an independent theory, 
readily applicable to  principle angles between infinite-dimensional subspaces as in \cite{{kja10}}. 
We still use the basis description of subspaces, but only in the proofs, using our previous matrix results. 
The statements are basis-free, e.g.,\ 
according to Theorem~\ref{cor:projectort} and Remark~\ref{rem:idempotent}, 
one possible form of the operator $T$ is a product of 
the orthogonal projector onto $\XP$ and the oblique projector that projects onto $\Y$ along $\XP$, 
where the singular values of $T$ determine the tangents of principal angles between subspaces $\X$ and $\Y$ and the subspace $\XP$ denotes the orthogonal complement to $\X$. 
If the subspaces $\X$ and $\Y$ are in generic position, i.e.,\ none of the PABS is zero or $\pi/2$,
the statement also becomes dimensionless, i.e., does not depend on $\dim(\X)$ or $\dim(\Y)$. 

The rest of the paper is organized as follows. We briefly review the concept and
some important properties of PABS, as well as related preliminaries, in Section~\ref{sec:2.1}.
The goal of this work is
explicitly constructing a family of matrices such that
their singular values are equal to the tangents of PABS.  
We form these matrices using bases for subspaces in Section~\ref{sec:tanorth} 
and projectors in Section~\ref{sec:tanproj}.
 
\section{Definition of PABS and other preliminaries}\label{sec:2.1}
In this section, we remind the reader the concept of PABS and 
some fundamental properties of PABS. 
We first recall that 
an acute angle between two unit vectors $x$ and $y$, 
i.e.,\ with $x^Hx = y^Hy = 1$, 
is defined as
\[\cos\theta(x, y) = |x^Hy|,\text{ where }0\leq\theta(x, y)\leq \pi/2.\]
This definition %of an acute angle between two vectors 
can be recursively extended to PABS; see, e.g., \cite{{bjorck73},{golub92},{hotelling36}}.
\begin{defin}
Let $\X\subset \C^{n}$ and $\Y\subset \C^{n}$ be subspaces with $\dim(\X)=p$ and $\dim(\Y)=q.$ 
Let $m=\min{(p,q)}.$
The principal angles 
$\Theta(\X,\Y) = \left[\theta_1,\ldots,\theta_{m}\right],$ 
where 
$\theta_k\in[0,\pi/2], \, k=1,\ldots,m,$ 
between $\X$ and $\Y$
are recursively defined by 
\[
%s_k=
\cos(\theta_{k})=\max_{x \in \X}\max_{y \in \Y}|x^Hy|=|x_k^H y_{k}|,
\]
subject to 
$
\|x\|=\|y\|=1,\,x^Hx_i=0,\,y^Hy_i=0,\,i=1,\ldots,k-1.
$
The vectors $\{x_1,\ldots, x_{m}\}$ and $\{y_1,\ldots, y_{m}\}$ are 
called the principal vectors. %of the pair of subspaces $\X$ and $\Y$.
\end{defin}

An alternative definition of PABS, from 
\cite{{bjorck73},{golub92}}, is based on the singular value decomposition
(SVD) and reproduced here as the following theorem. 

\begin{thm}\label{thm:2.1def}
Let the columns of matrices $X \in \C^{n \times p}$ and 
$Y \in \C^{n \times q}$ %($p\leq n$ and $q\leq n$) 
form  orthonormal bases for the subspaces $\X$ and $\Y$, 
correspondingly. Let the SVD of $X^HY$ be 
$
U \Sigma V^H,
$
where $U$ and $V$ are unitary matrices and $\Sigma$ is a $p \times q$ 
diagonal matrix with the real diagonal elements $s_1, \ldots, s_m$ 
in decreasing order with $m=\min(p,q)$. 
Then
$
\cos\Theta^{\uparrow}(\X,\Y)=S\left(X^HY\right)=
%S^{\downarrow}(X^HY)=
\left[s_1, \ldots, s_m\right],
$
where $\Theta^{\uparrow}(\X,\Y)$ denotes the vector of principal angles between 
 $\X$ and $\Y$ arranged in increasing order and $S(A)$ denotes the vector of
singular values of $A$.
Moreover, the principal vectors associated with this pair of subspaces are 
given by the first $m$ columns of 
$XU$ and $YV,$ correspondingly.
\end{thm}

Theorem \ref{thm:2.1def} implies that PABS are symmetric, i.e.\ 
$\Theta(\X,\Y)=\Theta(\Y,\X)$, and
unitarily invariant, i.e.,\ $\Theta(U\X,U\Y)=\Theta(\X,\Y)$ for 
 any unitary transformation $U$. % (Here $U\X=\{Ux: x\in \X\})$. 
Important properties of PABS have been established,
 for finite dimensional subspaces, e.g.,\ in~\cite{Ipsen95,knyazev06,sunste,sun87,wedin83}, and 
for infinite dimensional subspaces in~\cite{{Deutsch94},{kja10}}.
Relationships of principal
angles between $\X$ and $\Y$, and between their orthogonal complements  $\XP$ and $\YP$, correspondingly, are investigated in \cite{{Ipsen95},{knyazev06},{kja10}} as follows.
\begin{property}\label{pro:angle2}
Let $\Theta^{\downarrow}(\X,\Y)$ denote PABS arranged in decreasing order. Then:\\
(1) $\left[\Theta^{\downarrow}(\X,\Y),0,\ldots,0\right]
=\left[\Theta^{\downarrow}(\XP,\YP),0,\ldots,0\right],$\\ 
with $\max(n-\dim(\X)-\dim(\Y),0)$ zeros on the left and $\max(\dim(\X)+\dim(\Y)-n, 0)$ zeros on the right.\\ 
(2) $\left[\Theta^{\downarrow}(\X,\YP),0,\ldots,0\right]
=\left[\Theta^{\downarrow}(\XP,\Y),0,\ldots,0\right],$\\
with $\max(\dim(\Y)-\dim(\X),0)$ zeros on the left and $\max(\dim(\X)-\dim(\Y), 0)$ zeros on the right.\\
(3) $\left[\frac{\pi}{2},\ldots, \frac{\pi}{2},\Theta^{\downarrow}(\X,\Y)\right]
=\left[\frac{\pi}{2}-\Theta^{\uparrow}(\X,\YP),0,\ldots,0\right],$\\ 
with $\max(\dim(\X)-\dim(\Y),0)$ $\pi/2s$ on the left and $\max(\dim(\X)+\dim(\Y)-n, 0)$ zeros on the right. 
\end{property}

PABS are closely related to the Cosine-Sine Decomposition (CSD); e.g.,\ \cite{{es},{ps81},{pw94}}. 
Let $[\QX,\QXP]$ and $[\QY,\QYP]$ be unitary matrices with $\QX\in \C^{n \times p}$ and $\QY\in \C^{n \times q}$. Applying CSD to $[\QX,\QXP]^H[\QY,\QYP]$, we obtain
\begin{equation*}
[\begin{array}{ll} \QX & \QXP \end{array}]^H[\begin{array}{ll} \QY & \QYP \end{array}] 
=\left[ \begin{array}{c|l} \QX^H\QY & \QX^H\QYP \\ \hline
                               \QXP^H\QY & \QXP^H\QYP
             \end{array} \right] =\left[ \begin{array}{ll} U_{1} & \\             
              &  U_{2}  \end{array}  \right]
  D           
     \left[ \begin{array}{ll} V_{1} & \\ & V_{2}\end{array} \right]^H
\end{equation*}
with unitary matrices $U_{1},$ $U_2$, $V_1$, and $V_2$. The matrix $D$ has
 the following structure:
\begin{equation}\label{eqn:cs1}
D=\kbordermatrix{ &r&s&q-r-s&&n-p-q+r&s&p-r-s \\
  r    &I  &   &     & \vrule &  O &    &    \\
  s    &   & C &     & \vrule &        & S  &    \\
  p-r-s&   &   & O & \vrule&        &    &  I  \\ \hline
n-p-q+r&O&   &     & \vrule&  -I    &    &     \\
  s    &   & S &     & \vrule&        &-C  &    \\
 q-r-s &   &   &   I & \vrule&        &    & O },
\end{equation}
  where $C=\diag\left(\cos\left(\theta_{j_1}\right),\ldots,\cos\left(\theta_{j_s}\right)\right)$, and 
 $S=\diag(\sin(\theta_{j_1}),\ldots,\sin(\theta_{j_s}))$ such that $\theta_{j_k}\in(0,\pi/2)$ for $k=1,\ldots,s$,  
are all the principal angles between the subspaces $\Ra(\QY)$ and $ \Ra(\QX)$ located in the open interval $(0,\pi/2)$. 
Zero matrices of various sizes, not necessarily square, are denoted by $O$. 
 $I$ denotes the identity matrix. We
 may have different sizes of $I$ in~$D.$ In addition, it is possible to permute the first $q$ columns or the last 
$n-q$ columns of $D$, or the first $p$ rows or the last $n-p$ rows and 
to change the sign of any column or row to obtain the variants of the CSD.
 
The block sizes in the matrix $D$ are determined by the following decomposition of 
the space $ \C^n=\moo \oplus \moi \oplus \mio \oplus \mii \oplus \mm$ 
into an orthogonal sum of five subspaces, as in \cite{{halmos69},{kja10}}, 
defined via the column ranges $\X=\Ra\left(\QX\right) $ and 
$\Y=\Ra\left(\QY\right)$, and their orthogonal complements, $\XP$ and $\YP$, correspondingly, 
in the following way: 
\[\moo= \X \cap \Y, \enspace \moi=\X \cap \YP, 
\enspace \mio=\XP \cap \Y, \enspace
\mii=\XP \cap \YP.\] 
Namely, 
$
\dim(\moo)=r,\,
\dim(\mio)=q-r-s,\,
\dim(\moi)=p-r-s,$
and
$
\dim(\mii)=n-p-q+r,
$
according to \cite[Tables 1 and 2]{kja10}.
Decomposing $\mm=\mm_{\X} \oplus \mm_{\XP}=\mm_{\Y} \oplus \mm_{\YP}$, 
where
\begin{eqnarray*}
\mm_{\X}=\X \cap \left(\moo \oplus \moi\right)^{\perp}, &{}&
\mm_{\XP}=\XP \cap \left(\mio \oplus \mii\right)^{\perp},\\
\mm_{\Y}=\Y \cap \left(\moo \oplus \mio\right)^{\perp}, &{}&
\mm_{\YP}=\YP \cap\left(\moi \oplus \mii\right)^{\perp},
\end{eqnarray*}
we get $s=\dim(\mm_{\X})=\dim(\mm_{\Y})=\dim(\mm_{\XP})=\dim(\mm_{\YP})=\dim(\mm)/2.$
% Thus, we can represent $D$ in the following block form 
%\[
%\kbordermatrix{         &\dim(\moo)&\dim(\mm)/2&\dim(\mio)&&\dim(\mii)&\dim(\mm)/2&\dim(\moi) \\
 % \dim(\moo)    &I  &   &     & \vrule &  O^H &    &    \\
 % \dim(\mm)/2   &   & C &     & \vrule &        & S  &    \\
 % \dim(\moi)    &   &   & O & \vrule&        &    &  I  \\ \hline
 % \dim(\mii)    &O&   &     & \vrule&  -I    &    &     \\
 % \dim(\mm)/2   &   & S &     & \vrule&        &-C  &    \\
 % \dim(\mio)    &   &   &   I & \vrule&        &    & O^H  }.
 % \]

Finally, we extensively use the Moore-Penrose pseudoinverse; see, e.g.\ \cite{sunste}.
The Moore-Penrose pseudoinverse  $A^{\dagger}\in \C^{m\times n}$ of a matrix $A\in \C^{n\times m}$ 
satisfies the following $AA^{\dagger}A=A,\,  A^{\dagger}AA^{\dagger}=A^{\dagger},\,
(AA^{\dagger})^H=AA^{\dagger},\, (A^{\dagger}A)^H=A^{\dagger}A,$ and 
\begin{itemize}
	\item if $A = U\Sigma V^H$ is the SVD of $A$, 
	      then $A^{\dagger} = V\Sigma^{\dagger} U^H$;
  \item if $A$ has full column rank and $B$ has full row rank, 
       then $(AB)^{\dagger}=B^{\dagger}A^{\dagger}$. However, this formula does not hold
        in general;
  \item  $AA^{\dagger}$ is the orthogonal projector onto the range of $A$, and  
         $A^{\dagger}A$ is the orthogonal projector onto the range of $A^H$;
\item if $U$ and $V$ are unitary matrices then 
$%\[
(UAV)^{\dagger}=V^HA^{\dagger}U^H
$%\]
 for any matrix $A$;
\item let $A$, $B$, and $C$ be block matrices, such that
\[
A=\left[\begin{array}{ll} A_1           & O_A \end{array}\right],\,
B=\left[\begin{array}{l} B_1   \\ 
         O_B      \end{array}\right],\,
C=\left[\begin{array}{ll} C_1            & O_{12}\\ 
         O_{21}        & O_{2}\end{array}\right],
\]
where  $O_{\star}$ are various zero matrices. Then
\[
   A^{\dagger}=\left[\begin{array}{l} A_1^{\dagger}\\ 
                                O_{A}^H      \end{array}\right],\,  
B^{\dagger}=\left[\begin{array}{ll} B_1^{\dagger}           & O_{B}^H\end{array}\right],\,   
C^{\dagger}=\left[\begin{array}{ll} C_1^{\dagger}           & O_{21}^H\\ 
                                O_{12}^H       & O_{2}^H\end{array}\right].
\]          
\end{itemize}

%%%%%%%%%%%%%%%%%%%%%%%%%%%%%%%%%%%
%\section{Tangents of PABS}\label{sec:tan}

%In this chapter, we also present some known properties of PABS using new approaches.
 %%%%%%%%%%%%%%%%%%%%%%%%%%%%%%%%%%%%%%%%%%%
\section{\texorpdfstring{$\tan{\Theta}$}{tan{Theta}} in terms 
of the  bases of subspaces}\label{sec:tanorth}
Let the orthonormal columns of matrices $\QX$, $\QXP$, and $\QY$ span the 
subspaces $\X$, the orthogonal complement $\XP$ of $\X$, and $\Y$, correspondingly.  
Then $\cos\Theta(\X,\Y)=S(\QX^H\QY)$ and $\cos\Theta(\XP,\Y)=S(\QXP^H\QY)$ by Theorem \ref{thm:2.1def}. 
%One can obtain the tangents of PABS by using the
%sines and the cosines of PABS directly, 
%however, in some situations this does not reveal enough information about their properties. 
%%Let $\X$ and $\Y$ be two subspaces of dimension $k\leq n/2$. 
%The sines and cosines of PABS are obtained from the singular
%values of explicitly given matrices, which motivates us to construct $T$ as follows:
%if matrices $\QX$ and $\QY$ have orthonormal columns and $\QX^H\QY$ is invertible, 
%let $T=\QXP^H\QY\left(\QX^H\QY\right)^{-1}$, then $\tan\Theta(\X,\Y)$ 
%can be equal to the positive singular values of $T$. 
We begin with an example using 2D vectors. %in Figure~\ref{fig:may3}.
%\begin{figure}[ht]
%\begin{center}
%	\includegraphics[width=0.45\columnwidth]{cs_good1}
%\end{center}
%	\caption{The PABS in 2D.}
%	\label{fig:may3}
%\end{figure} 
Let
\[
\QX=\left[\begin{array}{l} 1\\ 0\end{array}\right],
\quad \QXP=\left[\begin{array}{r} 0\\ 1 \end{array}\right],\quad\text{and} \quad
 \QY=\left[\begin{array}{l} \cos\theta \\
 \sin\theta \end{array}\right],
 \] 
where $0\leq \theta < \pi/2$. 
Then,  $ \QXP^H\QY=\sin\theta$  and $\QX^H\QY=\cos\theta$.
Obviously, $\tan\theta$ is the singular value of $T=\QXP^H\QY\left(\QX^H\QY\right)^{-1}$. 
If $\theta=\pi/2$, then the matrix $\QX^H\QY$ is singular in this example.
Moreover,   if $\dim\X\neq\dim\Y$ the matrix $\QX^H\QY$ is rectangular, 
so we use its Moore-Penrose pseudoinverse 
to form our matrix $T=\QXP^H\QY\left(\QX^H\QY\right)^\dagger$. 
Now we are ready to prove our first main result.  
% that the intuition discussed above, suggesting to try $T=\QXP^H\QY\left(\QX^H\QY\right)^{\dagger}$, is correct.
% Moreover, we cover the case that $\QY$ does not necessarily have orthonormal columns. 

\begin{thm} \label{the: generalT1}
Let $[\QX,\QXP]$ be
a unitary matrix with $\QX\in~\C^{n\times p}$.
Let $\QY\in~\C^{n\times q}$ 
\begin{enumerate}
	\item have orthonormal columns, or
	\item be such that $\rank\left(\QY\right)=\rank\left(\QX^H\QY\right)$, where $\rank\left(\QX^H\QY\right)\leq p.$
\end{enumerate}
Then the positive
singular values of the matrix 
$T=\QXP^H\QY\left(\QX^H\QY\right)^{\dagger}$ satisfy %\in\C^{(n-p) \times p}$
 \begin{equation} \label{eqn:tan1}
 \tan\Theta(\X,\Y)=[\infty,\ldots,\infty,S_+(T),0,\ldots,0],
\end{equation} 
with $\min\left(\dim(\XP \cap \Y),\dim(\X \cap \YP)\right)$ 
$\infty's$ and $\dim( \X \cap \Y)$ zeros, where we denote $\Ra(X)=\X$ and $\Ra(Y)=\Y$. 

In case (ii), $\min\left(\dim(\XP \cap \Y),\dim(\X \cap \YP)\right)=0.$ 
  \end{thm}
\begin{proof}
(1) On the one hand, from equality \eqref{eqn:cs1}, we obtain  
\begin{equation*}
T=\QXP^H\QY\left(\QX^H\QY\right)^{\dagger}= U_2\left[\begin{array}{lll} 
                     O  & &\\
                       & SC^{-1} &\\
                       & & O  
     \end{array} \right]U_1^H,
\end{equation*}
Hence, $S_+(T)=\left[\tan(\theta_{j_1}),\ldots,\tan(\theta_{j_s})\right]$, 
where $0<\theta_{j_1}\leq \cdots \leq\theta_{j_s}<\pi/2$.
On the other hand, from~\eqref{eqn:cs1} we get $S\left(X^HY\right)=S\left(\diag(I,C,O)\right).$
where the identity block is $r$-by-$r$ and the  zero block is $(p-r-s)$-by-$(q-r-s)$.
By Theorem \ref{thm:2.1def}, 
 $\Theta(\Ra(\QX),\Ra(\QY))=[0,\ldots,0,\theta_{j_1},\ldots,\theta_{j_s},\pi/2,\ldots,\pi/2]$, 
 where %$\theta_{j_k} \in \left(0,\pi/2\right)$ for $k=1,\ldots,s$, and 
there are  $r=\dim\X\cap\Y$ zeros and $\min(q-r-s,p-r-s)=\min\left(\dim\XP\cap\Y,\dim\X\cap\YP\right)$  values $\pi/2$.  

%Next, we prove for the second case, where $Y$ does not necessarily have orthonormal columns, 
%but $\rank\left(\QY\right)=\rank\left(\QX^H\QY\right)\leq p$. 
(2) Let us denote the rank of $\QY$ be $t$, thus $ t\leq p$. 
Let the  SVD of $\QY$ 
be $U\Sigma V^H$, where $U$ is an $n \times n$ unitary matrix and $V$ is a $q \times q $ unitary matrix; 
$\Sigma$ is an $n \times q $ real diagonal matrix with diagonal 
entries ordered by decreasing magnitude.
Since rank($\QY$)=$t\leq q$, we can get a reduced SVD 
such that $\QY=U_t\Sigma_t V_t^H$. Only the $t$ column vectors 
of $U$ and the $t$ row vectors of $V^H$, corresponding to nonzero 
singular values are used, which means that $\Sigma_t$ is a 
$t$-by-$t$ invertible diagonal matrix. Based on the fact that 
the left singular vectors corresponding to the non-zero singular values of $Y$ span the range of $Y$, 
$%\[
\tan\Theta(\Ra(\QX),\Ra(\QY))=
 \tan\Theta(\Ra(\QX),\Ra(U_t)).
 $ %\]
 Since $\rank(X^HY)=t$, we have $\rank(\QX^HU_t\Sigma_t)=\rank(\QX^HU_t)=t.$ Let $T_1=\QXP^HU_t\left(\QX^HU_t\right)^{\dagger}$.
We have $
 \tan\Theta(\Ra(\QX),\Ra(U_t))=[S_+(T_1),0,\ldots,0].
$
It is worth noting that the angles between $\Ra(\QX)$ and $\Ra(U_t)$ are in $[0,\pi/2)$, since
$\QX^HU_t$ is full rank.

Our task is now to show that $T_1=T$. By direct computation, we have 
\begin{eqnarray*}
 T&=&\QXP^H\QY\left(\QX^H\QY\right)^{\dagger} \\
&=&\QXP^HU_t\Sigma_tV_t^H\left(\QX^HU_t\Sigma_tV_t^H\right)^{\dagger}\\
&=&\QXP^HU_t\Sigma_tV_t^H\left(V_t^H\right)^{\dagger}\left(\QX^HU_t\Sigma_t\right)^{\dagger}\\
&=&\QXP^HU_t\Sigma_t\left(\QX^HU_t\Sigma_t\right)^{\dagger} \\
&=&\QXP^HU_t\Sigma_t\Sigma_t^{\dagger}\left(\QX^HU_t\right)^{\dagger}\\
&=&\QXP^HU_t\left(\QX^HU_t\right)^{\dagger}.
\end{eqnarray*}
In two identities above we use the fact that if a matrix $A$ is of full column rank, 
and a matrix $B$ is of full row rank, then $(AB)^{\dagger}=B^{\dagger}A^{\dagger}.$ 
Hence, $\tan\Theta(\Ra(\QX),\Ra(\QY))=[S_+(T),0,\ldots,0]$ which
completes the proof for the second case.
\end{proof}
 
 \begin{remark}  If $\QX^H\QY$ has full rank, where $Y$ may have non-orthonormal columns, 
we can alternatively prove 
 Theorem \ref{the: generalT1} by constructing 
 $\QZ=\QX+\QXP T$ 
as in \emph{\cite[Theorem 2.4, p.252]{Stewart01}} and \emph{\cite[p.231-232]{sunste}}. 
 Since $\left(\QX^H\QY\right)^{\dagger}\QX^H\QY=I$, the following identities
$
\QY=P_\X\QY+P_\XP\QY=\QX\QX^H\QY+\QXP\QXP^H\QY=\QZ\QX^H\QY
$
imply $\Ra(\QY)\subseteq \Ra(\QZ)$. 
By direct calculation, we obtain that 
\[
\QX^H\QZ=\QX^H(\QX+\QXP T)=I
\an
\QZ^H\QZ=(\QX+\QXP T)^H(\QX+\QXP T)=I+T^HT. 
\]
 Thus,
$\QX^H\QZ\left(\QZ^H\QZ\right)^{-1/2}=(I+T^HT)^{-1/2}$
is Hermitian positive definite.
The~matrix $\QZ(\QZ^H\QZ)^{-1/2}$ by construction 
has orthonormal columns which span the space $\Z$. Moreover, we observe that
\[
S\left(\QX^H\QZ(\QZ^H\QZ)^{-1/2}\right)%=\Lambda\left((I+T^HT)^{-1/2}\right)
              =\left(1+S^2\left(T\right)\right)^{-1/2}.
\]
%where $\Lambda(\cdot)$ denotes the vector of eigenvalues. \\
Therefore,
$
\tan\Theta(\Ra(\QX),\Ra(\QZ))=[S_+(T),0,\ldots,0]
$
and
$\dim(\X)=\dim(\Z).$ By Theorem \ref{the: generalT1},
$%\[
\tan\Theta_{(0,\pi/2)}(\Ra(\QX),\Ra(\QY))=
\tan\Theta_{(0,\pi/2)}(\Ra(\QX),\Ra(\QZ)),
$%\]
where  all PABS in $(0,\pi/2)$ are denoted by $\Theta_{(0,\pi/2)}$.
In other words, the angles  in $(0,\pi/2)$ between subspaces $\Ra(\QX)$ and $\Ra(\QY)$ are the same as those between subspaces $\Ra(\QX)$ and $\Ra(\QZ)$. 

If $p=q,$ we have that $\Theta(\Ra(\QX),\Ra(\QY))=\Theta(\Ra(\QX),\Ra(\QZ))$. 
We note that this approach also gives us the explicit expression 
$P=\QXP^H\QY\left(\QX^H\QY\right)^{-1}$ for the matrix $P$ with $S(P)=\tan\Theta(\Ra(\QX),\Ra(\QY))$; 
cf. \emph{\cite[Theorem 2.4, p.252]{Stewart01}} and \emph{\cite[p.231-232]{sunste}}. 
%The case $p=q$ is considered in~\emph{\cite{ka09v1}}. 
%We extend their result to the case $q\leq p.$
\end{remark}

\begin{remark}
For the case $Y^HY\neq I,$ the condition $\rank\left(\QY\right)=
\rank\left(\QX^H\QY\right)$ is necessary in Theorem \ref{the: generalT1}. For example,
let 
\[
\QX= \left[\begin{array}{l} 1\\ 0\\0\end{array}\right], \quad
\QXP=\left[\begin{array}{ll} 0& 0\\ 1&0\\ 0&1\end{array}\right],\quad\an\quad
\QY=\left[\begin{array}{ll} 1&1\\ 0&1\\ 0&0\end{array}\right].
\]
Then, we have 
$\QX^H\QY=[\begin{array}{ll}1&1\end{array}]$ and 
$(\QX^H\QY)^{\dagger}=[\begin{array}{ll}1/2&1/2\end{array}]^H$. 
Thus,
$$ T=\QXP^H\QY\left(\QX^H\QY\right)^{\dagger}=[\begin{array}{ll}1/2& 0\end{array}]^H,\quad\an\quad s(T)=1/2.$$
On the other hand, we obtain that 
 $\tan\Theta(\Ra(\QX),\Ra(\QY))=0.$ From this example, we see that
 the result in Theorem \ref{the: generalT1} may fail in the case $\rank(\QY)>\rank\left(\QX^H\QY\right).$
 \end{remark} 

Due to the fact that PABS
are symmetric, i.e.,\ $\Theta(\X,\Y)=\Theta(\Y,\X)$, the matrix $T$ in 
Theorem \ref{the: generalT1} could be substituted with
$\QYP^HX(Y^HX)^{\dagger},$ where $[\QY,\QYP]$ is unitary. 
Moreover, the nonzero angles between the subspaces $\X$ and $\Y$ are
the same as those between the subspaces $\XP$ and $\YP$. Hence, $T$ can be presented
as $X^H\QYP(\QXP^H\QYP)^{\dagger}$ and $Y^H\QXP(\QYP^H\QXP)^{\dagger}$. Furthermore, 
for any matrix $T$, we have $S(T)=S(T^H)$, which implies that all conjugate transposes of $T$ are admissible. 

Let $\F$ denote a family of matrices, such that the singular values of the matrix
$T\in\F$ are the tangents of PABS.
The arguments above show that any of the formulas for $T$ 
in the first column of Table \ref{table:base1} can be used in Theorem \ref{the: generalT1}, case (i).

\begin{table}[ht]
	\caption{Different $T\in\F$ using orthonormal bases for $\X$, $\Y$, $\XP$, and $\YP$,
 see Theorem \ref{the: generalT1}, case (i);
	 the top two rows also applicable to non-orthonormal basis for $\Y$ if $q\leq p$, 
see Theorem \ref{the: generalT1}, case (ii).}
	\centering
		\begin{tabular}{|l|l|}
		\hline
		  $\QXP^H\QY\left(\QX^H\QY\right)^{\dagger}$& $P_{\XP}\QY\left(\QX^H\QY\right)^{\dagger}$\\ \hline
		  $(Y^HX)^{\dagger}Y^H\QXP$         &    $(Y^HX)^{\dagger}Y^HP_{\XP}$          \\ \hline \hline
			$\QYP^HX(Y^HX)^{\dagger}$ &  $P_{\YP} X(Y^HX)^{\dagger}$ \\ \hline			
			$X^H\QYP(\QXP^H\QYP)^{\dagger}$   &    $P_{\X}\QYP(\QXP^H\QYP)^{\dagger}$          \\ \hline
			$Y^H\QXP(\QYP^H\QXP)^{\dagger}$   &    $P_{\Y}\QXP(\QYP^H\QXP)^{\dagger}$   \\ \hline
			
			$(X^HY)^{\dagger}X^H\QYP$         &    $(X^HY)^{\dagger}X^HP_{\YP}$           \\ \hline
			$(\QYP^H\QXP)^{\dagger}\QYP^H\QX$ &    $(\QYP^H\QXP)^{\dagger}\QYP^HP_{\X}$  \\ \hline
		$(\QXP^H\QYP)^{\dagger}\QXP^H\QY$   &    $(\QXP^H\QYP)^{\dagger}\QXP^HP_{\Y}$  \\ \hline 
		\end{tabular}
		\label{table:base1}
	\end{table}
Using the fact that the singular values are invariant under unitary multiplications, 
we can also use $P_{\XP}Y(X^HY)^{\dagger}$ for $T$ in Theorem \ref{the: generalT1},
where $P_{\XP}$ is an orthogonal projector onto the subspace $\XP$. 
Thus, every matrix $T\in\F$ in the first column has its analog in $\F$ as in the second column in Table \ref{table:base1}, 
where, again, 
$P_{\X}$, $P_{\Y}$, and $P_{\YP}$ denote the orthogonal projectors onto 
the subspace $\X$, $\Y$, and $\YP$, correspondingly. 

Finally, if $Y^HY\neq I$ and $\rank\left(\QY\right)=\rank\left(\QX^H\QY\right)\leq p$, 
Theorem \ref{the: generalT1}, case (ii) holds. Using the above arguments, we see that 
any matrix $T$ in the top two rows in Table \ref{table:base1} belongs to the family $\F$ under these assumptions. 

\begin{remark}\label{rem: formular1}
From the proof of Theorem \ref{the: generalT1} case (i), we immediately derive  
$
\QXP^H\QY\left(\QX^H\QY\right)^{\dagger}=-(\QYP^H\QXP)^{\dagger}\QYP^H\QX,
$ 
and
$\QYP^HX(Y^HX)^{\dagger}=-(\QXP^H\QYP)^{\dagger}\QXP^H\QY,$
which demonstrates that some entries in Table \ref{table:base1} differ from each other only by a sign.
\end{remark}

Now we show that some of the matrices $T\in\F$ in Table \ref{table:base1} result in 
singular values $S(T)$ that also match the multiplicity of zeros in Theorem \ref{the: generalT1} case (ii).
\begin{cor}\label{cor:invertiblexy} 
Using the notation of Theorem \ref{the: generalT1}, let $Y$ be full rank and  $p=q$.
Let $P_{\XP}$ be an orthogonal projection onto the subspace $\Ra(\QXP)$. Then we have
$
\tan\Theta(\Ra(\QX),\Ra(\QY))=
S\left(P_{\XP}\QY\left(\QX^H\QY\right)^{-1}\right).
$
\end{cor}
\begin{proof}
%Since the singular values are invariant under unitary transform, it is easy to obtain
%$\left[S\left(\QXP^H\QY\left(\QX^H\QY\right)^{-1}\right), 0,\ldots,0\right]
%=S\left(P_{\XP}\QY\left(\QX^H\QY\right)^{-1}\right)$.
Theorem \ref{the: generalT1} case (ii) involves no $\infty$'s and, 
since  $Y$ is full rank and  $p=q$, the matrix $\QX^H\QY$ is invertible.  
Moreover, the number of the singular values of 
$P_{\XP}\QY\left(\QX^H\QY\right)^{-1}$ is $p=q$, which is the same as the number of PABS in this case.
\end{proof}

%We illustrate this approach by showing that 
The tangents of PABS also describe properties of blocks of  
the triangular matrix from the QR~factorization of a basis of the subspace $\X+\Y$. 
For brevity of presentation, we make simplifying assumptions on $\QX$ and $\QY$, 
avoiding the rank-revealing~QR. 
\begin{cor}\label{cor:generalT1}
Let  $X \in \C^{n \times p}$ and $Y \in \C^{n \times q}$ with $q\leq p$ be matrices of full rank,
 and also let $X^HY$ be full rank.
%Let the triangular part $R$ of the reduced QR factorization of the 
%	      matrix $[\QX\, \QY]$ be
%	     $
%	      R=\left[\begin{array}{lc} 
%                       R_{11} &R_{12}\\
%                       O   & R_{22}
%                       \end{array}\right].
%                    $       
  Let the QR factorization of $[\QX\, \QY]$ be
\begin{equation*}
[\begin{array}{ll} \QX &\QY\end{array}]=
[\begin{array}{ll}Q &Q_\perp\end{array}]\left[\begin{array}{lc} 
                       R_{11} &R_{12}\\
                       O   & R_{22}
                       \end{array}\right],
\end{equation*}    
where  $Q \in \C^{n \times p} \text{ and } Q_\perp \in \C^{n \times (n-p+1)}.$
Then 
\[\tan\Theta(\Ra(\QX),\Ra(\QY))=[S_+(R_{22}(R_{12})^{\dagger}),0,\ldots,0]\]
with $\dim(\Ra(\QX)\cap \Ra(\QY))$ zeros.
  \end{cor}
  \begin{proof}
Clearly, $\QX=QR_{11}$ and $\QY=QR_{12}+Q_\perp R_{22}.$
Since $\QX$ has full rank, we have $\Ra(Q)=\X$ and $R_{11}$ is invertible. 
Multiplying by $Q^H$ on both sides of equality for $\QY$, we get
$R_{12}=Q^H\QY.$ Moreover, $\rank\left(Q^H\QY\right)=\rank\left(\QX^H\QY\right)=\rank\left(\QY\right)$, since
$R_{11}$ is invertible and $X^HY=R_{11}^HQ^HY$. 
Multiplying the equality $\QY=QR_{12}+Q_\perp R_{22}$  by $Q_\perp^H$ gives $R_{22}=Q_\perp^H\QY$ 
and $R_{22}\left(R_{12}\right)^{\dagger}=Q_\perp^H\QY\left(Q^H\QY\right)^{\dagger}$. 
Theorem \ref{the: generalT1} case~(ii) holds with 
$\left[Q\,Q_\perp\right]$ substituting for $\left[\QX\,\QXP\right]$, 
since  we have $\rank\left(Q^H\QY\right)=\rank\left(\QY\right)=p$. 
 \end{proof}  

%%%%%%%%%%%%%%%%%%%%%%%%%%%%%%%%%%%%%%%%%%%%%%%%%%%%%%%%%%%%%%%%%
%Section 2.
%%%%%%%%%%%%%%%%%%%%%%%%%%%%%%%%%%%%%%%%%%%%%%%%%%%%%%%%%%%%%%%%
\section{\texorpdfstring{$\tan{\Theta}$}{tan{Theta}} in 
terms of projections onto subspaces}\label{sec:tanproj}
In the previous section, we rely on bases of subspaces to construct $T\in\F$. 
Now, we pursue a more basic geometric approach,  
representing subspaces by using orthogonal and oblique projectors, rather than matrices of their bases. 

\begin{thm}\label{cor:projectort}
Let $P_{\X}$, $P_{\XP}$ and $P_{\Y}$ be 
orthogonal projectors onto the subspaces $\X$, $\XP$ and $\Y$, correspondingly. 
Then the positive singular values 
$S_+(T)$ of the matrix $T=P_{\XP}\left(P_{\X}P_{\Y}\right)^{\dagger}$ satisfy 
$
\tan\Theta(\X,\Y)=[\infty,\ldots,\infty,S_+(T), 0, \ldots, 0],
$
where there are 
 $\min\left(\dim(\XP \cap \Y),\dim(\X \cap \YP)\right)$ 
$\infty's$ and $\dim( \X \cap \Y)$ zeros.
\end{thm}
\begin{proof}
Let the matrices $[\begin{array}{ll} \QX & \QXP \end{array}]$ 
and $[\begin{array}{ll} \QY & \QYP \end{array}]$ be unitary, where
$\Ra(\QX)=\X$ and $\Ra(\QY)=\Y.$ 
By direct calculation, we obtain
\begin{equation}
\label{eqn:productprojectors}                                                                                                                            
P_\X P_\Y=[\begin{array}{ll} 
\QX & \QXP \end{array}]
\left[\begin{array}{ll} \QX^H\QY & O \\
                          O  & O
      \end{array}\right]
      \left[\begin{array}{l} 
      \QY^H \\ \QYP^H
      \end{array}\right].
\end{equation}                                                                                                                   
%From Lemma \ref{cor:penrose1} and Corollary \ref{cor:penrose2}, 
Using properties of the Moore-Penrose pseudoinverse, reviewed in Section \ref{sec:2.1},
we have
\begin{eqnarray}
\label{eqn:inverseprojectors} 
(P_\X P_\Y)^{\dagger}&=&[\begin{array}{ll} 
\QY & \QYP \end{array}]
\left[\begin{array}{ll} (\QX^H\QY)^{\dagger} &O \\
                            O& O
      \end{array}\right]
      \left[\begin{array}{l} 
      \QX^H \\ \QXP^H
      \end{array}\right],
\end{eqnarray}
therefore
\[
T=P_{\XP}\left(P_{\X}P_{\Y}\right)^{\dagger}=[P_{\XP}Y(\QX^H\QY)^{\dagger}\quad O]\left[\begin{array}{l} 
      \QX^H \\ \QXP^H
      \end{array}\right].
\]
Thus, $S_+(T)=S_+(P_{\XP}Y(\QX^H\QY)^{\dagger})$, but the latter matrix is in Table \ref{table:base1}.
\end{proof}

\begin{remark}\label{rem:idempotent}
We note that the null space of the product $P_\X P_\Y$ is the orthogonal sum $\YP\oplus\left(\Y\cap\XP\right)$.
Thus, its orthogonal complement,  $\Y\cap\left(\Y\cap\XP\right)^{\perp},$ is thus the range of $(P_\X P_\Y)^{\dagger}$, which implies $P_{\Y}(P_\X P_\Y)^{\dagger}=(P_\X P_\Y)^{\dagger}=(P_\X P_\Y)^{\dagger}P_{\X}$ and so
$\left((P_\X P_\Y)^{\dagger}\right)^2=(P_\X P_\Y)^{\dagger}P_\X P_\Y(P_\X P_\Y)^{\dagger}=(P_\X P_\Y)^{\dagger}$. 
Therefore, we conclude that $(P_\X P_\Y)^{\dagger}$ is simply an idempotent that projects on $\Y$ along $\XP$.
\end{remark}

\begin{figure}[!htb]
\begin{center}
	\includegraphics[width=0.45\columnwidth]{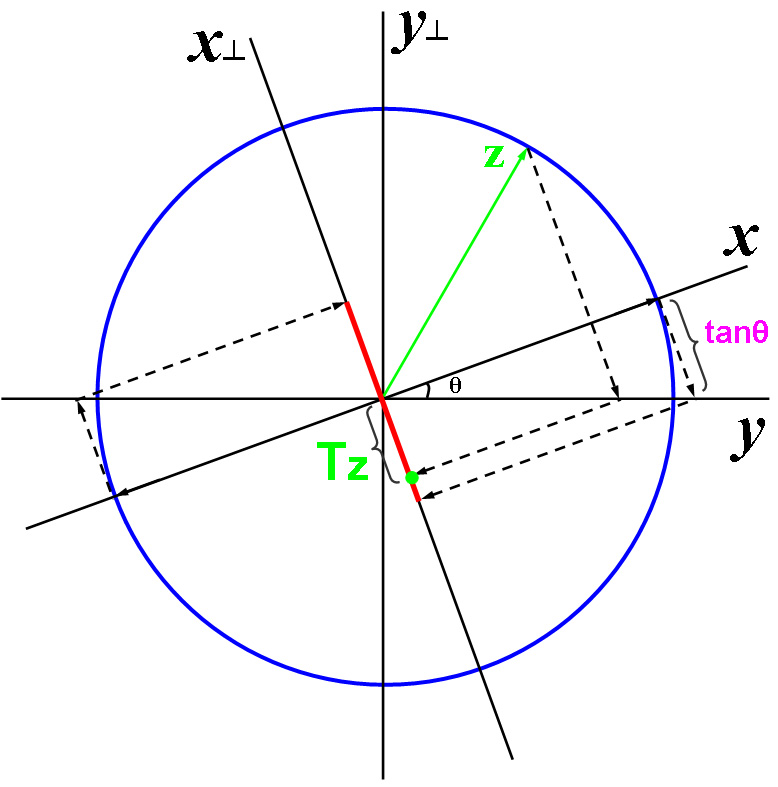}
\end{center}
%\label{fig:counters:fboxbottom}
\caption{Geometrical meaning of $T=P_{\XP}\left(P_{\X}P_{\Y}\right)^{\dagger}$.}\label{figureT}
\end{figure}

To gain geometrical insight into the action of $T$ from Theorem \ref{cor:projectort} as a linear transformation, 
in Figure \ref{figureT} we choose an arbitrary unit vector $z$. 
We project $z$ onto $\Y$ along $\XP$, then project onto the subspace $\XP$ 
which is interpreted as $Tz$. The red segment in the graph is 
the image of $T$ under all unit vectors. It is straightforward to see 
that $s(T)=\|T\|=\tan(\theta)$.

%\begin{remark}
%From representation \eqref{eqn:productprojectors}, we have
%\[                                                                                                                             
%S\left(P_\X P_\Y\right)=S\left(\left[\begin{array}{ll} X^HY &O \\
%                                              O& O\end{array}\right]\right)
%                       =\left[\cos\Theta(\X,\Y),0,\ldots,0\right].
%\]
%\end{remark}
%\begin{remark}
%From equality \eqref{eqn:complementyprojector} and Property \ref{pro:angle2}, we have
%\begin{eqnarray*}                                                                                                                 
%S\left(P_\XP P_\Y\right)&=&S\left(\left[\begin{array}{ll} \QXP^HY &O \\
%                                           O& O\end{array}\right]\right)
%                         =\left[\cos\Theta(\XP,\Y),0,\ldots,0\right]\\
%                         &=&\left[1,\ldots,1,\sin\Theta(\X,\Y),0,\ldots,0\right],
%\end{eqnarray*}
%where there are $\max(\dim(\Y)-\dim(\X),0)$ additional $1s$ on the right.
%\end{remark}

Using Property \ref{pro:angle2} and the fact that the principal angles are symmetric with
respect to the subspaces $\X$ and $\Y$, we observe that 
$P_{\XP}\left(P_{\X}P_{\Y}\right)^{\dagger}$ in Theorem \ref{cor:projectort}
can alternatively be substituted with 
$P_\YP(P_\Y P_\X)^{\dagger}$ or  $P_\X(P_\XP P_\YP)^{\dagger}$.
These expressions can be written in several forms, e.g.,\ Remarks \ref{rem:idempotent} implies 
\[
P_{\XP}\left(P_{\X}P_{\Y}\right)^{\dagger}=P_{\XP}P_{\Y}\left(P_{\X}P_{\Y}\right)^{\dagger}
=\left(P_\Y-P_\X\right)\left(P_\X P_\Y\right)^{\dagger}.
\]
From the proof of Theorem \ref{cor:projectort} we know that 
\[
P_{\XP}\left(P_{\X}P_{\Y}\right)^{\dagger}=[P_{\XP}Y(\QX^H\QY)^{\dagger}\quad 0]\left[\begin{array}{l} 
      \QX^H \\ \QXP^H
      \end{array}\right]=P_{\XP}Y(\QX^H\QY)^{\dagger}\QX^H.
 \]
      Similarly, we have $P_{\X}\left(P_{\XP}P_{\YP}\right)^{\dagger}=P_{\X}\QYP(\QXP^H\QYP)^{\dagger}\QXP^H.$
Then Remark \ref{rem: formular1} implies 
$P_{\XP}\left(P_{\X}P_{\Y}\right)^{\dagger}=-\left(P_{\X}\left(P_{\XP}P_{\YP}\right)^{\dagger}\right)^H.$
Using also Remark \ref{rem:idempotent} , we derive 
\begin{eqnarray*}
-\left(P_{\XP}\left(P_{\X}P_{\Y}\right)^{\dagger}\right)^H&=&P_\X(P_\XP P_\YP)^{\dagger}\\
&=&P_\X P_\YP(P_\XP P_\YP)^{\dagger}\\
&=&(P_\YP-P_\XP)(P_\XP P_\YP)^{\dagger}.
\end{eqnarray*}
Similar arguments justify the following chain of identities 
\begin{eqnarray*}
-\left(P_\Y(P_\YP P_\XP)^{\dagger}\right)^H&=&P_\YP (P_\Y P_\X)^{\dagger}\\
&=&P_\YP P_\X(P_\Y P_\X)^{\dagger}\\
&=&(P_\X-P_\Y)( P_\Y P_\X)^{\dagger}.
\end{eqnarray*}
This gives a variety of different possible choices of $T$ in Theorem~\ref{cor:projectort}.

Some of the formulas above can be simplified using the following lemma.
\begin{lemma}\label{thm:2.16}
Let $\Theta(\X,\Y)< \pi/2$. Then,   
\begin{enumerate}
\item if $\dim(\X)\leq\dim(\Y)$, we have
$P_{\X}(P_{\X}P_{\Y})^{\dagger}=P_{\X}$;
\item if $\dim(\X)\geq\dim(\Y)$, we have
$(P_{\X}P_{\Y})^{\dagger}P_{\Y}=P_{\Y}.$
\end{enumerate}
\end{lemma}
\begin{proof}
(1).  According to Remark \ref{rem:idempotent}, we get 
$P_{\X}(P_{\X}P_{\Y})^{\dagger}=P_{\X}P_{\Y}(P_{\X}P_{\Y})^{\dagger}$.
By the properties of the Moore-Penrose pseudoinverse, reviewed in Section \ref{sec:2.1}, 
$P_{\X}P_{\Y}(P_{\X}P_{\Y})^{\dagger}$ is the orthogonal projector onto the range 
 $\Ra\left(P_{\X}P_{\Y}\right)$ of $P_{\X}P_{\Y}$. The transformation $P_{\X}P_{\Y}$, if restricted to $\X$, is invertible, 
since $\Theta(\X,\Y)< \pi/2$ and $\dim(\X)\leq\dim(\Y)$; thus 
$\Ra\left(P_{\X}P_{\Y}\right)=\Ra\left(P_{\X}\right)$, which completes the proof of the first part.

Similarly, in part (2), $(P_{\X}P_{\Y})^{\dagger}P_{\Y}=(P_{\X}P_{\Y})^{\dagger}P_{\X}P_{\Y}$ 
is the orthogonal projector onto the range 
 $\Ra\left(\left(P_{\X}P_{\Y}\right)^H\right)=\Ra\left(P_{\Y}P_{\X}\right)=\Ra\left(P_{\Y}\right)$ 
since $\Theta(\X,\Y)< \pi/2$ and $\dim(\Y)\leq\dim(\X)$.
%Let the columns of matrices $\QX$ and $\QY$ form orthonormal bases for the subspaces $\X$ and $\Y$, correspondingly. 
%Since $\Theta(\X,\Y)< \pi/2$, the matrix $X^HY$ has full rank. 
% Using the fact that $X^HY(X^HY)^{\dagger}=I$ for $\dim(\X)\leq\dim(\Y)$ and combining identities \eqref{eqn:productprojectors} 
%and \eqref{eqn:inverseprojectors},
%we obtain the first statement. Identities \eqref{eqn:productprojectors} 
%and \eqref{eqn:inverseprojectors}, together with the fact that $(X^HY)^{\dagger}X^HY=I$ for $\dim(\X)\geq\dim(\Y)$,
%imply the second statement.
\end{proof}

\begin{remark}
From Lemma \ref{thm:2.16}, for the case $\dim(\X)\leq\dim(\Y)$ we have 
\[
P_\XP(P_\X P_\Y)^{\dagger}=(P_\X P_\Y)^{\dagger}-P_\X(P_\X P_\Y)^{\dagger}
                            =(P_\X P_\Y)^{\dagger}-P_\X.
\]
  Since the angles are symmetric, using the second statement in Lemma \ref{thm:2.16} we have 
   %$(P_{\XP}P_{\YP})^{\dagger}P_{\YP}=P_{\YP}$. Thus,
   $
   (P_{\XP}P_{\YP})^{\dagger}P_{\Y}=(P_{\XP}P_{\YP})^{\dagger}-P_{\YP}.
$   
    
  On the other hand, for the case $\dim(\X)\geq\dim(\Y)$ we obtain 
\[
(P_\X P_\Y)^{\dagger}P_\YP=(P_\X P_\Y)^{\dagger}-P_\Y\mspace{3mu}\an\mspace{3mu} (P_{\XP}P_{\YP})^{\dagger}P_{\X}=(P_{\XP}P_{\YP})^{\dagger}-P_{\XP}.
 \]
 To sum up, the following formulas for $T$ in Table \ref{table:orth12} can also
 be used in Theorem~\ref{cor:projectort}. An alternative proof for $T=(P_\X P_\Y)^{\dagger}-P_\Y$  is
 provided by Drma{\v{c}} in~\emph{\cite{drmac96}} for the particular case $\dim(\X)=\dim(\Y)$.
   
 \begin{table}[!htb]
\caption{Choices for $T\in\F$ with $\Theta(\X,\Y)< \pi/2$: left for $\dim(\X)\leq\dim(\Y)$; right for $\dim(\X)\geq\dim(\Y).$}
	\centering
		\begin{tabular}{|l|l|}
		\hline
			$(P_\X P_\Y)^{\dagger}-P_\X$   & 	$(P_\X P_\Y)^{\dagger}-P_\Y$  \\ \hline			
		   $(P_\XP P_\YP)^{\dagger}-P_\YP$    & 	$(P_\XP P_\YP)^{\dagger}-P_\XP$      \\ \hline
					\end{tabular}
		\label{table:orth12} 
\end{table}
  \end{remark}
  
Our choice of the space $\H=\C^{n}$
may appear natural to the reader familiar with the matrix theory, 
but in fact is somewhat misleading. 
The principal angles (and the corresponding principal vectors) between 
the subspaces $\X\subset\H$ and $\Y\subset\H$ are exactly the same 
as those between the subspaces $\X\subset\X+\Y$ and $\Y\subset\X+\Y$, 
i.e.,\ we can reduce the space $\H$ to the space $\X+\Y\subset\H$
without changing PABS. 
This reduction changes the definition of the subspaces $\XP$ and $\YP$
and, thus, of the matrices $\QXP$ and $\QYP$
that column-span the subspaces $\XP$ and $\YP$. 
All our statements that use the subspaces $\XP$ and $\YP$
or the matrices $\QXP$ and $\QYP$ therefore have their new 
analogs, if the space $\X+\Y$ substitutes for~$\H$. 
The formulas $P_\XP=I-P_\X$ and $P_\YP=I-P_\Y$, 
look the same, but the identity operators are 
different in the spaces $\X+\Y$ and $\H$.

\def\refname{

\end{document}